\documentclass[a4paper,10pt]{article}
\usepackage{amsmath,amssymb}
\usepackage{amsthm}
\usepackage{bm}
\usepackage[dvipdfmx]{graphicx}
\usepackage{color}
\usepackage[abs]{overpic}
\usepackage[pstarrows]{pict2e}
\usepackage{tabularx}

\theoremstyle{plain} 
\newtheorem{thm}{\indent\bf Theorem}[section] 
\newtheorem{lmm}[thm]{\indent\bf Lemma}
\newtheorem{crl}[thm]{\indent\bf Corollary}
\newtheorem{prp}[thm]{\indent\bf Proposition}
\theoremstyle{definition} 

\newtheorem{rem}[thm]{\indent\bf Remark}

\makeatletter
\newcommand\sectiondot[1]{%
  \expandafter\def\csname @seccntfmt@#1\endcsname##1{%
    \csname the##1\endcsname.\quad
  }
}
\sectiondot{section}
\newcommand\sectionpunct[2]{%
  \expandafter\def\csname @seccntfmt@#1\endcsname##1{%
    \csname the##1\endcsname#2%
  }%
}
\renewcommand\@seccntformat[1]{\@ifundefined{@seccntfmt@#1}%
  {\csname the#1\endcsname\quad}%
  {\csname @seccntfmt@#1\endcsname{#1}}%
}
\sectiondot{subsection}
\renewcommand\sectionpunct[2]{%
  \expandafter\def\csname @seccntfmt@#1\endcsname##1{%
    \csname the##1\endcsname#2 %
  }%
}
\@addtoreset{equation}{section}

\renewcommand\section{\@startsection{section}{1}{\parindent}%
                                     {-3.5ex \@plus -1ex \@minus -.2ex}%
                                     {-2.3ex \@plus -.2ex}%
                                     {\normalfont\normalsize\bfseries}}
\renewcommand\subsection{\@startsection{subsection}{2}{\parindent}%
                                     {-3.25ex\@plus -1ex \@minus -.2ex}%
                                     {-1.5ex \@plus -.2ex}%
                                     {\normalfont\normalsize\bfseries}}
\renewcommand\paragraph{\@startsection{paragraph}{4}{\parindent}%
                                     {-3.25ex\@plus -1ex \@minus -.2ex}%
                                     {-1.5ex \@plus -.2ex}%
                                     {\normalfont\normalsize\bfseries}}
\makeatother

\newcommand{\colR}[1]{#1}
\newcommand{\colC}[1]{#1}
\newcommand{\colRR}[1]{#1}
\newcommand{\colCC}[1]{#1}
\newcommand{\colRRR}[1]{#1}

\newcommand{\colRRRR}[1]{#1}

\newcommand{\textR}[1]{#1}
\newcommand{\BB}[1]{\mathbb{#1}}
\newcommand{\pd}[2]{\frac{\partial#1}{\partial#2}}

\newcommand{\n}{\mathbf{n}}
\renewcommand{\v}{\mathbf{v}}
\newcommand{\wt}[1]{\widetilde{#1}}
\newcommand{\ta}[1]{\widetilde{A_{#1}}}
\newcommand{\tk}[1]{\widetilde{\kappa_{#1}}}
\newcommand{\tv}[1]{\widetilde{\mathbf{v}_{#1}}}
\newcommand{\ma}{\mathcal{A}}
\newcommand{\hs}[1]{\hspace{#1pt}}

\newcommand{\p}{\hspace*{-5pt}\strut^+}
\newcommand{\m}{\hspace*{-5pt}\strut^-}

\title{Height functions on Whitney umbrellas}
\author{
Toshizumi \textsc{FUKUI}
\footnote{Department of Mathematics, Faculty of Science, Saitama
University, 255 Shimo-Okubo, Sakura-ku, Saitama 338-8570, Japan.\newline
e-mail: \texttt{tfukui@rimath.saitama-u.ac.jp}}
~and
Masaru \textsc{HASEGAWA}
\footnote{Department of Mathematics, Faculty of Science, Saitama
University, 255 Shimo-Okubo, Sakura-ku, Saitama 338-8570, Japan.\endgraf
e-mails:
\texttt{mhasegawa@mail.saitama-u.ac.jp}, 
\texttt{masaruhasegawa@hotmail.co.jp}}}
\date{}
\begin{document}
\maketitle

\footnote{2000 \textit{Mathematics Subject Classification}.
Primary 53A05; Secondary 58K05, 58K35.}
\footnote{\textit{Key words and phrases}.
Whitney umbrella, cross-cap, height function.}

\begin{abstract}      
 We study the singularities of the members of the family of height
 functions on Whitney umbrellas, which is also known as cross-caps, and
 show that the family of the height functions is a versal unfolding.
 Moreover, we study local intersections of a Whitney umbrella with a
 hyperplane through its singular point. 
\end{abstract}

\section{Introduction}
H. Whitney \cite{Whitney} 
\colRR{investigated} singularities which are unavoidable
by small perturbation
\colRR{, and found the singularity type, which is}
\colRR{called Whitney umbrellas (or cross-caps).}
This singularity \colRR{type} is very important\colRR{,} 
since it is the only singularity of a map of \colRR{a} surface to
3-dimensional Euclidean space 
which is stable under small deformations.
\colRR{It is natural to investigate Whitney umbrellas as a subject of
differential geometry, and there are several articles
\colRRR{\cite{BW1998, FH2012, FN2005, GGS2000, GS1986, HHNUY2012, NT2007,
Oliver2011, ST2012, Tari2007, West1995}} in this
direction. }

In \colRR{the} authors' previous work \cite{FH2012}, we introduce some
differential geometric ingredients (principal curvatures, ridge,
sub-parabolic points, etc.) for Whitney umbrellas, and investigate
singularities and versal unfoldings of distance squared functions on
Whitney umbrellas in terms of these ingredients.
To develop the differential geometric ingredients, we consider the
double oriented blowing-up (\cite[example (a) in p. 221]{Hironaka})
\begin{equation*}
 \label{eq:double_blow-up}
 \wt{\Pi}:\BB{R}\times S^1\to\BB{R}^2,\quad
 (r,\theta)\mapsto(r\cos\theta,r\sin\theta),
\end{equation*}
and a map, which is usually called a blow-up,
\begin{equation*}
 \label{eq:blow-up}
 \Pi:\mathcal{M}\to\BB{R}^2,\quad
 [(r,\theta)]\to(r\cos\theta,r\sin\theta), 
\end{equation*}
where $\mathcal{M}$ denotes the quotient space of $\BB{R}\times S^1$
with identification $(r,\theta)\sim(-r,\pi+\theta)$.
We remark that $\mathcal{M}$ is topologically a M\"obius strip. 
There is a well-defined unit normal vector via $\wt{\Pi}$,
thus we obtain the information on the asymptotic behavior of the
principal curvatures and the principal directions.
Moreover, we also obtain the configuration of the parabolic
\colRR{line (or ridge lines, or sub-parabolic lines)} on $\mathcal{M}$
near the exceptional set $E=\Pi^{-1}(0,0)$. 

In this paper, we investigate the relationships between singularities
of height functions on Whitney umbrellas and
these differential geometric ingredients for Whitney umbrellas, and the
versality of the family of height functions.
We also study intersections of a Whitney umbrella and hyperplanes
through its singular point.
Height functions are a fundamental tool for the study of the
differential geometry of submanifolds.
Several authors have studied the geometrical properties of submanifolds
in Euclidean space by analyzing the singularities of 
height functions on submanifolds (see, for example,
\cite{BGT1995,MRR1995,P1971}).

The intersection of a surface and a hyperplane is intimately related to
the height function on the surface.
In particular, for a regular surface in $\BB{R}^3$, the contact between
the surface and a hyperplane is measured by the singularity type of the
height function on the surface.
For instance, when the contact of a point $p$ is 
\colR{$A_1$} type (
\colRR{i.e.,} the height function has an $A_1$ singularity), the
intersection of the surface and the hyperplane at $p$, which is the
tangent plane at $p$, is locally an isolated point or a pair of
transverse curves (see, for example, \cite{BGT1995}).

In Section~2, we 
\colR{recall} the differential geometric ingredients developed in
\cite{FH2012}. 
In Section~3, we investigate singularities and versal unfoldings of
heigh functions on Whitney umbrella in terms of the
differential geometry which is discussed in Section 2.
Theorem \ref{thm:main1} shows that the relationship of the singularities
of the height functions and the differential geometric ingredients for
Whitney umbrellas.
In Section~4, we study the local intersections of a Whitney umbrella and
hyperplanes through 
\colR{the} singular point.
\colR{
To 
\colRR{detect local singularity types of the} intersections of a surface
with hyperplanes, we must investigate how the zero set of the height
function on the surface is mapped to the surface.
For a regular surface, it is trivial, since the contact of the surface
between hyperplanes is measured by the singularities of the height
function on the surface. 
On the other hand, for a Whitney umbrella, we must investigate how the
zero set of the height function on the Whitney umbrella contacts with
the double point locus $\mathcal{D}$ in the parameter space.
Theorem~\ref{thm:main2} shows that if the hyperplane $\mathcal{P}$
through the Whitney umbrella singularity does not contain the tangent
\colRR{line} to
$\mathcal{D}$ on the surface then the contact of the Whitney umbrella
and $\mathcal{P}$ at the singularity is measured by the type of
singularities of the height function, and that if $\mathcal{P}$ contains
the tangent to $\mathcal{D}$ on the surface then the contact of the
Whitney umbrella and $\mathcal{P}$ is measured by the type of
singularities of the height function and the contact between the zero
set of the height function and $\mathcal{D}$ in parameter space.
Hence, Theorem \ref{thm:main2} is considered to be an analogous theorem
of Montaldi \cite{Montaldi1986} for Whitey umbrellas.}
 
\section{The differential geometry for Whitney umbrella}
In this section, we 
\colR{recall} the differential geometry of Whitney
umbrellas, which is obtained in \cite{FH2012}.

A smooth map germ $g:(\BB{R}^2,{\bf 0})\to(\BB{R}^3,{\bf 0})$ is a {\it
Whitney umbrella} if $g$ is $\ma$-equivalent to the map germ
$(\BB{R}^2,{\bf 0})\to(\BB{R}^3,{\bf 0})$ defined by
$(u,v)\mapsto(u,uv,v^2)$. 
Here, map germs $f_1$, $f_2:(\BB{R}^2,{\bf 0})\to(\BB{R}^3,{\bf 0})$ are
{\it $\ma$-equivalent} if there exist diffeomorphism map germs
$\phi_1:(\BB{R}^2,{\bf 0})\to(\BB{R}^2,{\bf 0})$ and
$\phi_2:(\BB{R}^3,{\bf 0})\to(\BB{R}^3,{\bf 0})$ such that
$f_1\circ\phi_1=\phi_2\circ f_2$.

For a regular surface in $\BB{R}^3$, at any point on the surface we can
take the $z$-axis as the normal direction. After suitable rotation if
necessary, the surface can be locally expressed in Monge form:
\[
 (u,v)\mapsto(u,v,\frac12(k_1u^2+k_2v^2)+O(u,v)^3).
\]
For a Whitney umbrella, we have the similar normal form 
\colRR{to} Monge form.
\begin{prp}
 \label{prop:normalform}
 Let $g:(\BB{R}^2,{\bf 0})\to(\BB{R}^3,{\bf 0})$ be a
 Whitney umbrella.
 Then there are $T\in \mathrm{SO}(3)$ and a diffeomorphism  
 $\phi:(\BB{R}^2,{\bf 0})\to(\BB{R}^2,{\bf 0})$ so that 
 \begin{equation}
  \label{eq:normal_form}
  T\circ g\circ\phi(u,v)=\biggl(u,\
 uv+\sum_{i=3}^k\frac{b_i}{i!}v^i+O(u,v)^{k+1},\
 \sum_{m=2}^kA_m(u,v)
 +O(u,v)^{k+1}\biggr), 
 \end{equation}
 where 
 \[
  A_m(u,v)=\sum_{i+j=m}\frac{a_{ij}}{i!j!}u^i v^j\quad(a_{02}\ne0).
 \]
\end{prp}
This result was first obtained in \cite{West1995}.
\colRR{A} proof can also be found in \cite{FH2012}.
Without loss of generality, we may assume that $a_{02}>0$ (see
\cite[Introduction]{HHNUY2012}).
The form \eqref{eq:normal_form} with such an assumption is called the
{\it normal form of Whitney umbrella}.

We suppose that $g:(\BB{R}^2,{\bf 0})\to(\BB{R}^3,{\bf 0})$ is given in the
normal form of Whitney umbrella.
We shall introduce some differential geometric ingredients (the unit
normal vector, principal curvatures, and principal vectors etc.) for
Whitney umbrellas (see \cite{FH2012} for details).

\paragraph{Unit normal vectors.}
We denote the unit normal vector of $g$ by $\n(u,v)$.
The unit normal vector $\wt{\n}=\n\circ\wt{\Pi}$ in the coordinates $(r,\theta)$
is expressed as follows:
\begin{equation*}
 \label{eq:unit_normal}
 \wt{\n}(r,\theta)=\frac{(0,-a_{11}\cos\theta-a_{02}\sin\theta,\cos\theta)}{\ma(\theta)}+O(r),
\end{equation*}
where $\ma(\theta)=\sqrt{\cos^2\theta+(a_{11}\cos\theta+a_{02}\sin\theta)^2}$.

\paragraph{Principal curvatures and principal directions.}
We denote by $\kappa_1$ and  $\kappa_2$ the principal curvatures.
We also denote by $\v_1$ and  $\v_2$ the principal vectors associated with
$\kappa_1$ and  $\kappa_2$, respectively.
The principal curvatures $\tk{i}=\kappa_i\circ\wt{\Pi}$ in the coordinates $(r,\theta)$
are expressed as follows:
\begin{align}
 \label{eq:blow_k1}
 \tk{1}(r,\theta)&=\frac{
 (a_{20}\cos^2\theta-a_{02}\sin^2\theta)\sec\theta}{\ma(\theta)}+O(r),\\
 \label{eq:blow_k2}
 \tk{2}(r,\theta)&=\frac1{r^2}\left(\frac{a_{02}\cos\theta}{\ma(\theta)^3}+O(r)\right).
\end{align}
The (unit) principal vectors $\tv{i}$ in the coordinates $(r,\theta)$
are expressed as follows:
\begin{align*}
 \tv{1}(r,\theta)&=(\sec\theta+O(r))\pd{}{r}
 +\left(\frac{-2\wt{A_{3v}}(\theta)+b_3\wt{A_{2v}}(\theta)\sin\theta\tan\theta}{2a_{02}}+O(r)\right)\pd{}{\theta},\\
 \tv{2}(r,\theta)&=\frac1{r^2}\left[\left(\frac{\sin\theta}{\ma(\theta)}r+O(r^2)\right)\pd{}{r}
 +\left(\frac{\cos\theta}{\ma(\theta)}+O(r)\right)\pd{}{\theta}\right].
\end{align*}
Here, $\wt{A_{2v}}(\theta)=A_{2v}|_{(u,v)=(\cos\theta,\sin\theta)}$,
$\wt{A_{3v}}(\theta)=A_{3v}|_{(u,v)=(\cos\theta,\sin\theta)}$ and so on.

\paragraph{Gaussian curvature.}
\label{sec:Gaussian}
Since the Gaussian curvature $K$ is the product of the principal
curvatures, it does not depend on the choice of the unit normal vector. 
From \eqref{eq:blow_k1} and \eqref{eq:blow_k2}, 
the Gaussian curvature $\wt{K}$ of $g$ in the coordinates $(r,\theta)$
is expressed as follows:
\begin{align}
 \label{eq:Gaussian}
 \wt{K}(r,\theta)=K\circ\Pi(r,\theta)
 =\frac1{r^2}\left(\frac{a_{02}(a_{20}\cos^2\theta-a_{02}\sin^2\theta)}{{\mathcal{A}(\theta)}^4}+O(r)\right).
\end{align}
By this expression, we say that a point $(0,\theta_0)$ on $\mathcal{M}$
is {\it elliptic}, {\it hyperbolic}, or {\it parabolic point over
Whitney umbrella} if 
$r^2\wt{K}(0,\theta_0)$ is positive, negative, or
zero, respectively.
We often omit the phrase ``over Whitney umbrella'' if
no confusion is possible from the context.

\paragraph{Parabolic lines.}
It is shown in \cite{West1995} that Whitney umbrellas are generically
classified into two types:
the {\it elliptic Whitney umbrella}, whose parabolic line, in the parameter
space, has an $A_1^-$ singularity (locally two transversally intersecting
curves); and the {\it hyperbolic Whitney umbrella}, whose parabolic line
has an $A_1^+$ singularity (locally an isolated point).
At the transition between two types, there is a Whitney umbrella whose
parabolic line has an $A_2$ singularity (locally a cusp).
Such a Whitney umbrella is called a {\it parabolic Whitney umbrella}
(see \cite{NT2007,Oliver2011}).
 The parabolic line of $g$ in $\mathcal{M}$ is expressed as follows:
\label{sec:parabolic}
\begin{align*}
 \begin{split}
  \label{eq:parabolic_line}
  0&=\text{\small{$
  a_{02}(a_{20}\cos^2\theta-a_{02}\sin^2\theta)
  +\Bigl[(a_{30}a_{02}+a_{12}a_{20})\cos^3\theta$}}\\
  &\text{\small{$
  \quad+(2a_{21}a_{02}+a_{03}a_{20}-a_{20}a_{11}b_3)\cos^2\theta\sin\theta
  -a_{02}(a_{03}-a_{11}b_3)\sin^3\theta\Bigr]r+O(r^2)$}}.
 \end{split}
\end{align*}
From this form, we obtain the following proposition providing criteria
for the type of Whitney umbrella.
\begin{prp}
 \label{prop:type_of_Whitney_umbrella}
\begin{enumerate}
 \vspace{-6pt}\setlength{\parskip}{0cm}\setlength{\itemsep}{0cm} 
 \item The map $g$ is an elliptic $($resp. hyperbolic$)$ Whitney
       umbrella if and only if $a_{20}>0$ $($resp. $<0)$.
 \item The map $g$ is a parabolic Whitney umbrella if and only if
       $a_{20}=0$ and $a_{30}\ne0$.
\end{enumerate}
\end{prp}
\begin{figure}[ht]
 \begin{minipage}[t]{0.33\textwidth}
  \centering
  \includegraphics[width=0.7\textwidth]{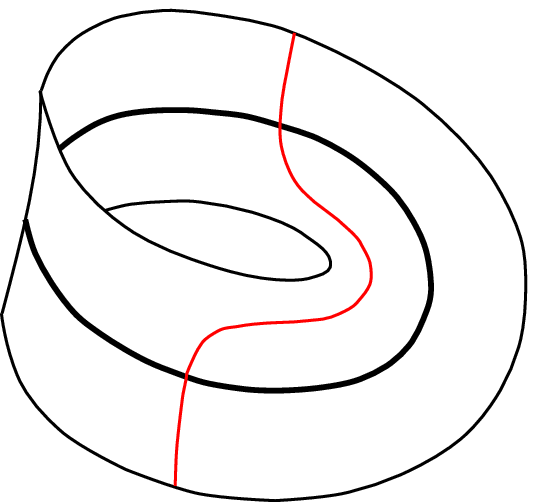}
 \end{minipage}\hfill
 \begin{minipage}[t]{0.33\textwidth}
  \centering
  \includegraphics[width=0.7\textwidth]{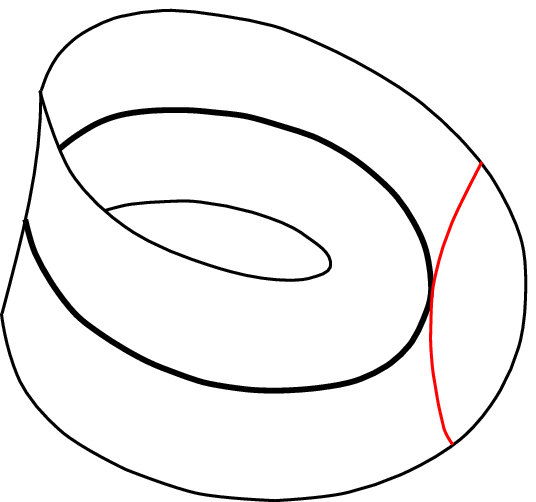}
 \end{minipage}\hfill
 \begin{minipage}[t]{0.33\textwidth}
  \centering
  \includegraphics[width=0.7\textwidth]{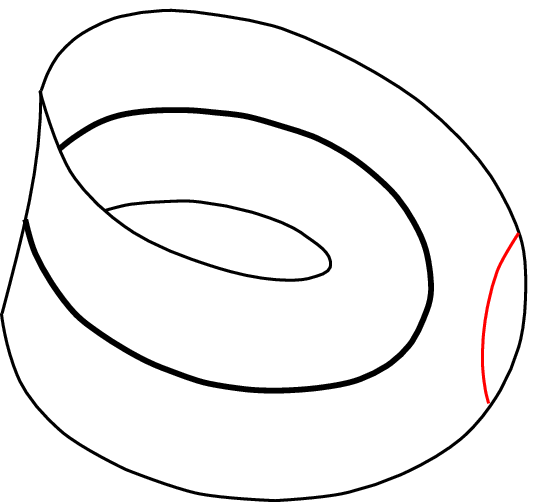}
 \end{minipage}
 \caption{Parabolic lines on $\mathcal{M}$: elliptic Whitney umbrella left,
 parabolic Whitney umbrella center, hyperbolic Whitney umbrella right.}
\end{figure}
\paragraph{Ridge and sub-parabolic points.}
Ridge points of regular surfaces in $\BB{R}^3$ were first studied by
Porteous \cite{P1971} in terms of singularities of distance squared
functions. 
The ridge points are points where one principal curvature has an
extremal value along lines of the same principal curvature.
The ridge points are also correspond to singularities on the focal
surface.

The directional derivative of $\tk{1}$ in the direction $\tv{1}$ is
expressed as follows: 
\begin{equation*}
 \tv{1}\tk{1}(r,\theta)
 =\frac{\left(6\ta{3}(\theta)\cos\theta-b_3\ta{2v}(\theta)\sin^3\theta\right)\sec^3\theta}{\ma(\theta)}+O(r),\\
\end{equation*}
We say that a point $(r_0,\theta_0)$ is a {\it ridge point relative to
$\tv{1}$} if $\tv{1}\tk{1}(r_0,\theta_0)=0$.
In particular, if the ridge point $(r_0,\theta_0)$ is in the exceptional
set $E=\Pi^{-1}(0,0)$ (that is, $r_0=0$), the point is said to be the
{\it ridge point over Whitney umbrella}.
Moreover, a point $(r_0,\theta_0)$ is a {\it $n$-th order ridge point
relative to $\tv{1}$} if ${\tv{1}}^{(k)}\tk{1}(r_0,\theta_0)=0$ $(1\leq
k\leq n)$ and ${\tv{1}}^{(n+1)}\tk{1}(r_0,\theta_0)\ne0$, where
${\tv{1}}^{(k)}\tk{1}$ denotes the $k$-times directional derivative of
$\tk{1}$ in $\tv{1}$.
The twice directional derivative of $\tk{1}$ in $\tv{1}$ is expressed as follows:
\begin{align*}
 {\tv{1}}^2\tk{1}(r,\theta)
 =\frac{\left(\ta{3vv}(\theta)\cos\theta-b_3\ta{2v}\right)\tv{1}\tk{1}(0,\theta)
 \sec^2\theta}
 {a_{02}}
 +\frac{\Gamma(\theta)\sec^5\theta}{a_{02}\ma(\theta)}+O(r)
\end{align*}
where
\begin{align*}
 \Gamma(\theta)&
 =\text{\small{$
 24a_{02}\ta{4}(\theta)\cos^2\theta
 -12a_{02}\left({\ta{2}(\theta)}^2+\cos^2\theta\sin^2\theta\right)(a_{20}\cos^2\theta-a_{02}\sin^2\theta)$}}\\\nonumber
 &\quad\text{\small{$
 -3\left(2\ta{3v}(\theta)\cos\theta-b_3\ta{2v}(\theta)\sin^2\theta\right)^2
 -a_{02}b_4\ta{2v}(\theta)\cos\theta\sin^4\theta$}}.\\\nonumber
\end{align*}

Sub-parabolic points of regular surfaces in $\BB{R}^3$ were first
studied in terms of folding maps in \cite{BW1991} and
\cite{Wilkinson1991}.
The sub-parabolic points are points where one principal curvature has an
extremal value along lines of the other principal curvature.
The sub-parabolic points are also correspond to parabolic points on the
focal surface.

The directional derivative of $\tk{1}$ in the direction $\tv{2}$ is
expressed as follows:  
\begin{equation}
 \label{eq:sub-parabolic}
 \tv{2}\tk{1}(r,\theta)
 =\frac1{r^2}\left(-\frac{2a_{02}\left(\ta{2}(\theta)\ta{2v}(\theta)+\cos^2\theta\sin\theta\right)}{\ma(\theta)^4}+O(r)\right).
\end{equation}
We say that a point $(r_0,\theta_0)$ is a {\it sub-parabolic point relative
to $\tv{2}$} if $r^2\tv{2}\tk{1}(r_0,\theta_0)=0$.
In particular, if the sub-parabolic point $(r_0,\theta_0)$ is in the
exceptional set $E$, 
the point is said to be the
{\it sub-parabolic point over Whitney umbrella}.
\begin{rem}
 The map $g$ is a parabolic Whitney umbrella if and only if 
 $g$ has a unique parabolic point $(r,\theta)=(0,0)$ which is not a
 ridge point over Whitney umbrella.
 In this case, $(0,0)$ is a sub-parabolic point relative to $\tv{2}$
 over Whitney umbrella (cf.~Proposition~\ref{prop:type_of_Whitney_umbrella}). 
\end{rem}

\paragraph{Focal conics.}
\colR{In \cite{FH2012},} we introduce the focal conic as the counterpart of the focal point.
We define a family of distance squared functions $D:(\BB{R}^2,{\bf 0})\times
\BB{R}^3\to\BB{R}$ on a Whitney umbrella $g$ by
$D(u,v,x,y,z)=\|(x,y,z)-g(u,v)\|^2$. 
It is shown in \cite
{FH2012} that the function $D$ has an $A_k$
singularity $(k\geq 2)$ at $(0,0)$ for $(x,y,z)$ if and only if
$(x,y,z)$ is on a conic in the normal plane.
Such a conic is called a {\it focal conic}.
The focal conic of the normal form of Whitney umbrella is given by the
equation
\begin{equation*}
 y^2+2a_{11}y z-(a_{20}a_{02}-{a_{11}}^2)z^2+a_{02}z=0.
\end{equation*}
This implies that the focal conic is an ellipse, hyperbola, or parabola
if and only if $a_{20}<0$, $a_{20}>0$, or $a_{20}=0$, respectively.
Hence, we obtain the following criteria for the (generic) type of
Whitney umbrellas as corollary of Proposition
\ref{prop:type_of_Whitney_umbrella}. 
\begin{crl}
 The map $g$ is an elliptic $($resp. hyperbolic$)$ Whitney umbrella if and
 only if its focal conic is a hyperbola $($resp. ellipse$)$.
\end{crl}
\begin{figure}[h]
 \begin{minipage}[t]{0.33\textwidth}
  \centering
  \includegraphics[width=0.6\textwidth]{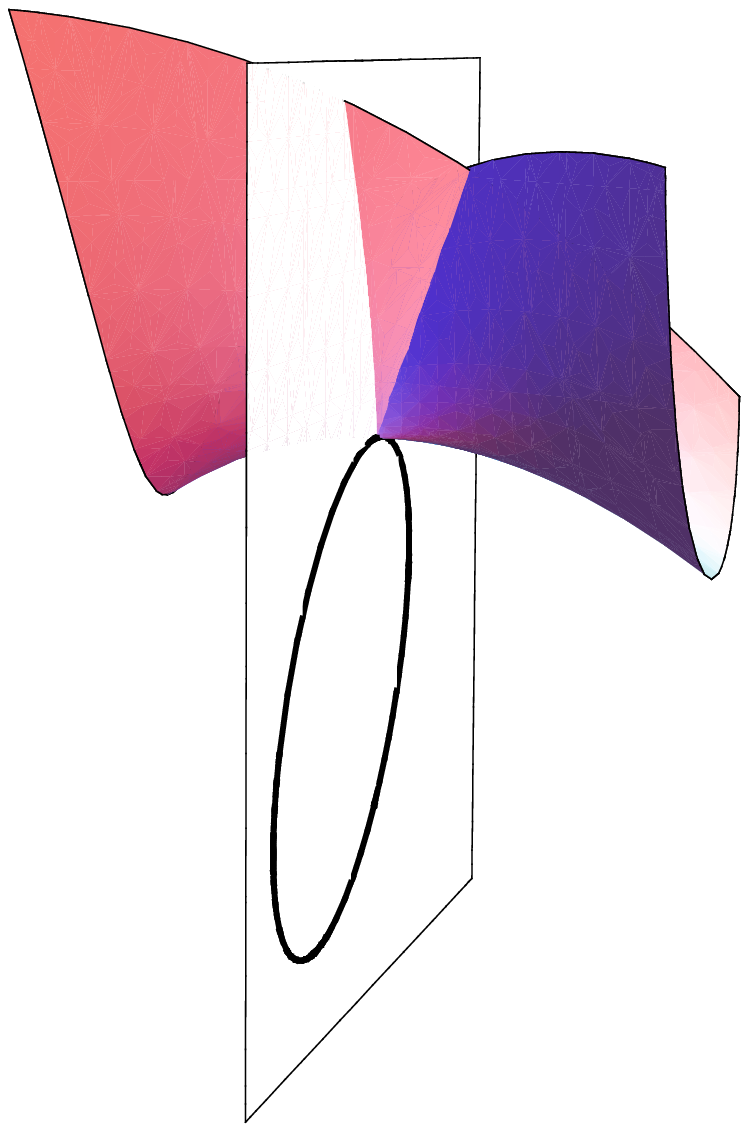}
 \end{minipage}\hfill
 \begin{minipage}[t]{0.33\textwidth}
  \centering
  \includegraphics[width=0.6\textwidth]{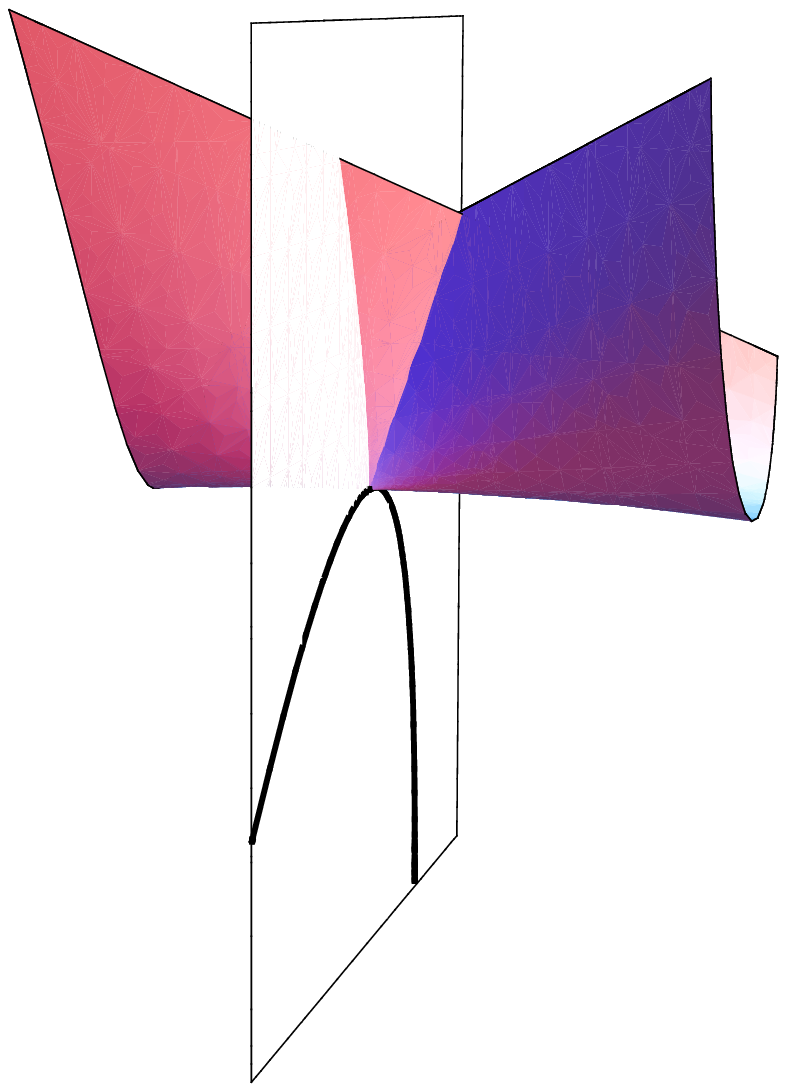}
 \end{minipage}\hfill
 \begin{minipage}[t]{0.33\textwidth}
  \centering
  \includegraphics[width=0.55\textwidth]{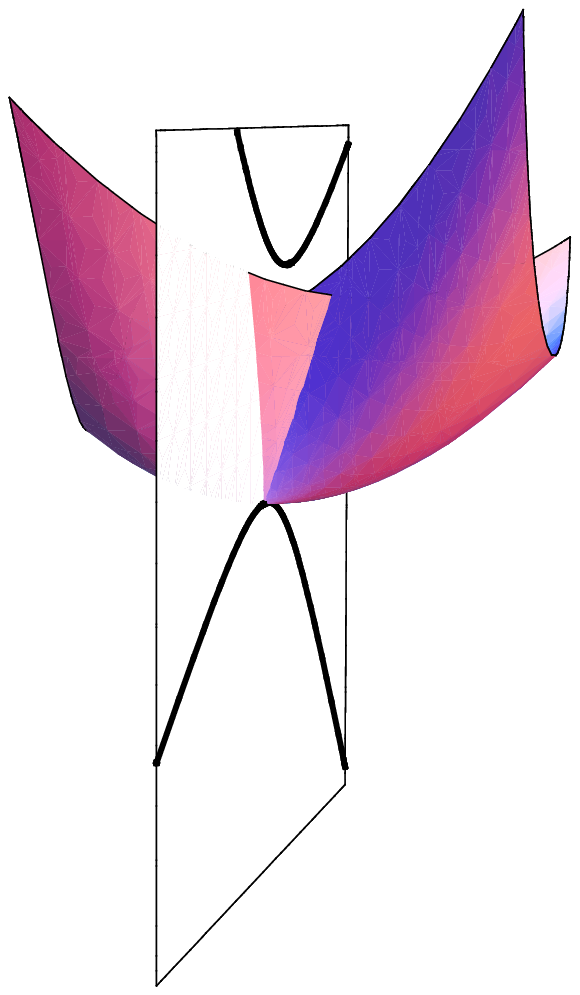}
 \end{minipage}
 \caption{The three types of focal conics: focal ellipse left, focal parabola center,
 focal hyperbola right. }
 \label{fig:focal_conic}
\end{figure}

\section{Height functions on a Whitney umbrella}
\label{sec:height}
 In this section, we investigate singularities and
 $\mathcal{R}^+$-versal unfoldings of height functions on Whitney
 umbrellas.

 Firstly, we recall the definition of unfoldings and 
 \colRR{its} $\mathcal{R}^+$-versality.
 See \colRR{\cite[Section~8 and 19]{Arnold1}} for details. See also
 \colCC{\cite[Section~3]{Wall1}}. 
 Let $f:(\BB{R}^n,{\bf 0})\to(\BB{R},{\bf 0})$ be a smooth function germ. We
 say that a smooth function germ $F:(\BB{R}^n\times\BB{R}^r,{\bf
 0})\to(\BB{R},{\bf 0})$ is an {\it $r$-parameter unfolding} of $f$ if
 $F|\BB{R}^n\times\{{\bf 0}\}=f$.
 Let $F:(\BB{R}^n\times\BB{R}^r,{\bf
 0})\to(\BB{R},{\bf 0})$ be an unfolding of $f$. The unfolding $F$ is an
 {\it $\mathcal{R}^+$-versal unfolding} of $f$ if for every unfolding
 $G:(\BB{R}^n\times\BB{R}^s,{\bf 0})\to(\BB{R},{\bf 0})$ of $f$ there
 exist a smooth map germ $\Phi:(\BB{R}^n\times\BB{R}^s,{\bf
 0})\to(\BB{R},{\bf 0})$ with $\Phi({\bf u},{\bf 0})={\bf u}$, a smooth
 map germ $\varphi:(\BB{R}^s,{\bf 0})\to(\BB{R}^r,{\bf 0})$, and a
 smooth function germ $\lambda:(\BB{R}^s,{\bf 0})\to(\BB{R},{\bf 0})$
 such that $G({\bf u},{\bf x})=F(\Phi({\bf u},{\bf x}),\varphi({\bf
 x}))+\lambda({\bf x})$, where ${\bf u}\in\BB{R}^n$ and ${\bf x}\in
 \BB{R}^s$.
%
 The set of smooth function germs $(\BB{R}^n,{\bf 0})\to\BB{R}$ is
 denoted by $\mathcal{E}_n$.
 The set $\mathcal{E}_n$ is the local ring with the unique maximal ideal 
 $\mathcal{M}_n=\{f\in\mathcal{E}_n\,;\,f({\bf 0})=0\}$.
 A smooth function germ $f$ is said to be {\it right-$k$-determined} if
 all function germs with a given $k$-jet are right-equivalent. 
 If $f\in\mathcal{M}_n$ is right-$k$-determined, then an
 unfolding $F:(\BB{R}^n\times\BB{R}^r,{\bf 0})\to(\BB{R},{\bf 0})$ of
 $f$ is an $\mathcal{R}^+$-versal unfolding if and only if
 \begin{equation*}
  \label{eq:Rver}
  \mathcal E_n=
  \left<\pd{f}{u_1},\ldots,\pd{f}{u_n}\right>_{\mathcal E_n}
  +
  \left<
  \left.\pd{F}{x_1}\right|_{\BB{R}^n\times\{{\bf 0}\}},\ldots,
  \left.\pd{F}{x_r}\right|_{\BB{R}^n\times\{{\bf 0}\}}
  \right>_{\BB{R}}+\langle 1\rangle_{\BB{R}}
  +\langle u,v\rangle_{\mathcal{E}_n}^{k+1}
 \end{equation*}
 (cf.~\cite[p.~149]{Mar1}).
 
Let $g:(\BB{R}^2,{\bf 0})\to(\BB{R}^3,{\bf 0})$ be a smooth map which
defines a surface in $\BB{R}^3$.
We set a family of functions
\begin{equation*}
 \label{eq:height}
 H:(\BB{R}^2\times S^2,({\bf 0},{\bf w}_0))\to \BB{R}
\end{equation*}
by $H(u,v,{\bf w})=\langle g(u,v), {\bf w}\rangle$, where
$\langle\cdot,\cdot\rangle$ is the 
\colRR{Euclidean inner product in} $\BB{R}^3$. 
We define $h(u,v)=H(u,v,{\bf w}_0)$. The function $h$ is the {\it height
function} on $g$ in the direction ${\bf w}_0$ and the family $H$ is a
$2$-parameter unfolding of $h$.

The following theorem is our main theorem.
\begin{thm}
 \label{thm:main1}
 Let $g:(\BB{R}^2,{\bf 0})\to(\BB{R}^3,{\bf 0})$ be given in the normal
 form of Whitney umbrella, \colR{and let $h$ be}
 the height function on $g$ in the direction ${\bf
 w}_0=\pm\wt {\n}(0,\theta_0)$, where $\theta_0\in(-\pi/2,\pi/2]$.
 Then $h$ has a singularity at $(0,0)$. Moreover, we have the following
 assertions.
 \begin{enumerate}
  \vspace{-6pt}\setlength{\parskip}{0cm}\setlength{\itemsep}{0cm} 
  \item The height function $h$ has an $A_1^+$
        $($resp. $A_1^-)$ singularity at $(0,0)$ if and only if
        $(0,\theta_0)$ is an elliptic $($resp. hyperbolic$)$ point over
        Whitney umbrella. In this case, $H$ is an
        $\mathcal{R}^+$-versal unfolding of $h$.
  \item The height function $h$ has an $A_2$ singularity at $(0,0)$ if
        and only if $(0,\theta_0)$ is a parabolic point and not a ridge
        point relative to $\tv{1}$ over Whitney umbrella. In this case,
        $H$ is an $\mathcal{R}^+$-versal unfolding of $h$.
  \item The height function $h$ has an $A_3^+$
        $($resp. $A_3^-)$ singularity at $(0,0)$ if 
        and only if $(0,\theta_0)$ is a parabolic point and first order
        ridge point relative to $\tv{1}$ over Whitney umbrella, and
        ${\tv{1}}^2\tk{1}(0,\theta_0)>0$ $($resp. $<0)$. In this
        case, $H$ is an $\mathcal{R}^+$-versal unfolding
        of $h$ if and only if $g$ is an elliptic Whitney umbrella.
  \item The height function $h$ has an $A_k$ singularity $(k\geq 4)$ at
        $(0,0)$ if and only if $(0,\theta_0)$ is a parabolic point and
        second or higher order ridge point relative to $\tv{1}$ over Whitney
        umbrella.
  \item The height function $h$ does not have 
        \textR{a singularity of type $D_4$ or worse $($that is, the Hessian matrix of $h$ never be of
        rank $0)$} at $(0,0)$.
 \end{enumerate}
\end{thm}
\begin{proof}
 We now remark that
 \[
  {\bf w}_0=\pm\wt{\n}(0,\theta_0)
 =\pm\left(0,\frac{-a_{11}\cos\theta_0-a_{02}\sin\theta_0}{\ma(\theta_0)},\frac{\cos\theta_0}{\ma(\theta_0)}\right).
 \]
 We have $g_u(0,0)=(1,0,0)$ and $g_v(0,0)=(0,0,0)$, and thus
 $H_u(0,0,{\bf w})=H_v(0,0,{\bf w})=0$ if and only if ${\bf w}$ is on
 the $yz$-plane
 .
 Hence, $h$ has a singularity at $(0,0)$.

 Next, we prove the criteria for the type of singularities of $h$.
 Calculations show that
 \[
 h_{uu}(0,0)=\pm\frac{a_{20}\cos\theta_0}{\ma(\theta_0)},\quad
 h_{uv}(0,0)=\mp\frac{a_{02}\sin\theta_0}{\ma(\theta_0)},\quad
 h_{vv}(0,0)=\pm\frac{a_{02}\cos\theta_0}{\ma(\theta_0)},
 \]
 so that the determinant of the Hessian matrix of $h$ at $(0,0)$ is given by
 \[
  \frac{a_{02}(a_{20}\cos^2\theta_0-a_{02}\sin^2\theta_0)}{{\ma(\theta_0)}^2} 
  =r^2\ma(\theta_0)^2\wt{K}(0,\theta_0).
 \]
 From this, $h$ has an $A_1^+$ (resp. $A_1^-$) singularity at
 $(0,0)$ if and only if $(0,\theta_0)$ is an elliptic
 (resp. hyperbolic) point over Whitney umbrella.
 Moreover, $h$ has 
 \textR{a singularity of type of $A_2$ or worse singularity at $(0,0)$} if and
 only if $(0,\theta_0)$ is a parabolic point over Whitney umbrella.
 When $(0,\theta_0)$ is a parabolic point over Whitney umbrella we have
 $\cos\theta_0\ne0$, that is, $h_{vv}(0,0)\ne 0$.
 Thus, $h$ does not have a 
 \textR{singularity of type $D_4$ or worse at $(0,0)$.}

 We assume that $h$ has an $A_k$ $(k\geq 2)$ singularity at $(0,0)$.
 Now we regard $S^2$ as the set of unit vectors in $\BB{R}^3$ and we
 write $(x,y,z)={\bf w}\in S^2$.
 For any ${\bf w}\in S^2$, we have $x^2+y^2+z^2=1$.
 We may assume that $z\ne0$.
 We have $z=\pm\sqrt{1-x^2-y^2}$.
 We set
 \[
 b=\frac{h_{11}}{h_{02}},\quad
 c=\frac{h_{21}(h_{02})^2-2h_{12}h_{11}h_{02}+h_{03}(h_{11})^2}{2(h_{02})^3},
 \]
 where
 \[
  h_{ij}=\frac{\partial^{i+j}h}{\partial u^i\partial v^j}(0,0).
 \]
 Replacing by $v$ by $v-bu-cv^2$ and writing down $H$ as
 \begin{equation}
  \label{eq:monge_height_unfolding}
   H=u x+\sum_{i+j=2}^4\frac1{i!j!}c_{ij}(x,y)u^iv^j+O(u,v)^5,
 \end{equation}
 we have
 \begin{equation}
  \label{eq:monge_height_function}
  h=\frac12c_{02}^0v^2
   +\frac16(c_{30}^0u^3+3c_{12}^0u v^2+c_{03}^0v^3)
   +\sum_{i+j=4}\frac1{i!j!}c^0_{ij}(x,y)u^iv^j+O(u,v)^5,
 \end{equation}
 where
 \[
  c_{ij}(x,y)=\frac{\partial^{i+j}H}{\partial u^i \partial v^j}(0,0,x,y)
 \]
 and $c^0_{ij}=c_{ij}(x_0,y_0)$.
 Note that $c^0_{02}=h_{vv}(0,0)\ne0$.
 The expansion \eqref{eq:monge_height_function} implies that
 $h$ has $A_k$ singularity at $(0,0)$ if and only if the following
 conditions holds:
\begin{itemize}
 \item $A_2:c_{30}^0\ne0$;
 \item $A_3^+\,(\text{resp.\,}A_3^-):c_{30}^0=0,\,c_{02}^0c_{40}^0>0\,
       (\text{resp.\,}<0)$;
 \item $A_{\geq4}:c_{30}^0=c_{40}^0=0$.
\end{itemize} 
 Straightforward calculations show that
 $c_{30}^0=0$ is equivalent to $\tv{1}\tk{1}(0,\theta_0)=0$, and that
 if 
 $c_{30}^0=0$ then
 $c_{02}^0c_{40}^0=a_{02}\ma(\theta_0){\tv{1}}^2\tk{1}(0,\theta_0)\cos\theta_0$.
 This conclude the proof of the criteria the type of singularities of $h$.

 Next, we proceed to check the versality of $H$.
 We skip the proofs of $(1)$ and $(2)$, since the proofs are similar to
 that of $(3)$. 
 We assume that $h$ has an $A_3$ singularity at $(0,0)$, and we
 consider the expansions \eqref{eq:monge_height_unfolding} and
 \eqref{eq:monge_height_function}. 
 We 
 \colR{note} that $(0,\theta_0)$ is now a parabolic point over Whitney
 umbrella.
 Therefore, from \eqref{eq:Gaussian} we have $a_{20}\ge 0$.
 In order to show the $\mathcal{R}^+$-versality of $H$, we need to
 verify that the following equality holds:
 \begin{equation}
  \label{eq:RverA3}
  \mathcal E_2=
  \left< \pd{h}{u},\pd{h}{v}\right>_{\mathcal E_2}
  +
  \left<
  \left.\pd{H}{x}\right|_{\BB{R}^2\times \{{\bf w}_0\}},
  \left.\pd{H}{y}\right|_{\BB{R}^2\times \{{\bf w}_0\}}
  \right>_{\BB{R}}
  +\langle 1 \rangle_{\BB{R}}
  +\langle u,v\rangle_{\mathcal{E}_2}^5.
 \end{equation}
 The coefficients of $u^iv^j$ of functions appearing in
 \eqref{eq:RverA3} are given by the following table:
\[
 \begin{array}{@{\hs{8}}c@{\hs{8}}|
  @{\hs{8}}c@{\hs{8}}c@{\hs{8}}|
   @{\hs{8}}c@{\hs{8}}c@{\hs{8}}c@{\hs{8}}|
   @{\hs{8}}c@{\hs{8}}c@{\hs{8}}c@{\hs{8}}c@{\hs{8}}|
   @{\hs{8}}c@{\hs{8}}}\hline
   &&&&&&&&&&\\[-10pt]
   & u & v & u^2 & uv & v^2 & u^3 & u^2v & uv^2 & v^3 & u^4 \\[4pt]\hline
   &&&&&&&&&&\\[-10pt]
   H_x & \fbox{1} & 0 &
   \text{\footnotesize{$\alpha_{20}/2$}} & \alpha_{11} &
   \text{\footnotesize{$\alpha_{02}/2$}} & 
   \text{\footnotesize{$\alpha_{30}/6$}} & 
   \text{\footnotesize{$\alpha_{21}/2$}} & 
   \text{\footnotesize{$\alpha_{12}/2$}} & 
   \text{\footnotesize{$\alpha_{03}/6$}} & 
   \text{\footnotesize{$\alpha_{40}/24$}} \\[6pt]
   H_y & 0 & 0 & 
   \text{\footnotesize{$\beta_{20}/2$}} & \beta_{11} &
   \text{\footnotesize{$\beta_{02}/2$}} & 
   \text{\footnotesize{$\beta_{30}/6$}} & 
   \text{\footnotesize{$\beta_{21}/2$}} & 
   \text{\footnotesize{$\beta_{12}/2$}} & 
   \text{\footnotesize{$\beta_{03}/6$}} & 
   \text{\footnotesize{$\beta_{40}/24$}} \\[6pt]
   \hline 
   &&&&&&&&&&\\[-10pt]
   h_u & 0 & 0 &
   \text{\footnotesize{$c_{30}^0/2$}} & 0 & 
   \text{\footnotesize{$c_{21}^0/2$}} & 
   \fbox{\text{\footnotesize{$c_{40}^0/6$}}} & 
   \text{\footnotesize{$c_{31}^0/2$}} & 
   \text{\footnotesize{$c_{22}^0/2$}} & 
   \text{\footnotesize{$c_{13}^0/6$}} & 
   \text{\footnotesize{$c_{50}^0/24$}} \\[4pt] 
   h_v & 0 
   & \fbox{$c_{02}^0$} & 0 & 
   c_{12}^0 & 
   \text{\footnotesize{$c_{03}^0/2$}} & 
   \text{\footnotesize{$c_{31}^0/6$}} & 
   \text{\footnotesize{$c_{22}^0/2$}} & 
   \text{\footnotesize{$c_{13}^0/2$}} & 
   \text{\footnotesize{$c_{04}^0/6$}} & 
   \text{\footnotesize{$c_{41}^0/24$}} \\[6pt]
   \hline 
   &&&&&&&&&&\\[-10pt]
   u h_u & 0 & 0 & 0 & 0 
   & 0 &
   \text{\footnotesize{$c_{30}^0/2$}} & 0 & 
   \text{\footnotesize{$c_{12}^0/2$}} & 0 &
   \fbox{\text{\footnotesize{$c_{40}^0/6$}}}\\[4pt]
   u h_v & 0 & 0 & 
   \colR{0} & \fbox{$c_{02}^0$} & 0 & 0 
   & c_{12}^0 &
   \text{\footnotesize{$c_{03}^0/2$}} & 0 &
   \text{\footnotesize{$c_{31}^0/6$}}\\[4pt]
   v h_v & 0 & 0 & 0 & 0 
   & \fbox{$c_{02}^0$} & 0 & 0 
   & c_{12}^0 &
   \text{\footnotesize{$c_{03}^0/2$}} & 0 \\[4pt]
   \hline 
   &&&&&&&&&&\\[-10pt]
   u^2 h_v & 0 & 0 & 0 & 0 & 0 & 0 
   & \fbox{$c_{02}^0$} & 0 & 0 & 0 
   \\[4pt]
   u v h_v & 0 & 0 & 0 & 0 & 0 & 0 & 0 & 
   \fbox{$c_{02}^0$} & 0 & 0 \\[4pt]
   v^2 h_v & 0 & 0 & 0 & 0 & 0 & 0 & 0 & 0 
   & \fbox{$c_{02}^0$} & 0 \\[4pt]\hline
 \end{array}
 \]
 \[
  \begin{array}{@{\hs{8}}c@{\hs{8}}|
   @{\hs{8}}c@{\hs{8}}|
    @{\hs{8}}c@{\hs{8}}@{\hs{8}}c@{\hs{8}}@{\hs{8}}c@{\hs{8}}@{\hs{8}}c@{\hs{8}}@{\hs{8}}c@{\hs{8}}}\hline
   &&&&&&\\[-10pt]
    & u^iv^j\ (i+j\le3) &u^4&u^3v&u^2v^2&uv^3&v^4\\[4pt]\hline
    &&&&&&\\[-10pt]
    u^3h_v & 0 & 0 
    & \fbox{$c_{02}^0$} & 0 & 0 & 0 \\
   u^2vh_v & 0 & 0 & 0 
    & \fbox{$c_{02}^0$} & 0 & 0 \\
   uv^2h_v & 0 & 0 & 0 & 0 
    & \fbox{$c_{02}^0$} & 0 \\
   v^3h_v & 0 & 0 & 0 & 0 & 0 
    & \fbox{$c_{02}^0$} \\[6pt]\hline
  \end{array}
 \]
 Here, $\alpha_{ij}=(\partial c_{ij}/\partial x)(x_0,y_0)$,
 $\beta_{ij}=(\partial c_{ij}/\partial y)(x_0,y_0)$. 
 \textR{We remark that the boxed entries are non-zero.}
 \colR{Calculations give}
 \begin{equation*}
  \label{eq:coefficients_of_H}
\textR{\beta_{20}=\frac{r^2\ma(\theta_0)^4\tv{2}\tk{1}(0,\theta_0)\sec^3\theta_0}{a_{02}}.}
 \end{equation*}
 \colR{It follows that} the matrix presented by this table
 is full rank, that is, \eqref{eq:RverA3} holds if and only if
 $(0,\theta_0)$ is not a sub-parabolic point relative to $\tv{2}$ over
 Whitney umbrella. 
 It now follows from \eqref{eq:Gaussian} and \eqref{eq:sub-parabolic}
 that $(0,\theta_0)$ not being a sub-parabolic point over Whitney
 umbrella is equivalent to $a_{20}\ne0$.  
 Therefore, from Proposition~\ref{prop:type_of_Whitney_umbrella}, we
 have completed the proof of the necessary and sufficient condition for
 $H$ to be $\mathcal{R}^+$-versal unfolding of $h$.
%
%
\end{proof}
%

\begin{rem}
 Theorem~\ref{thm:main1} implies that the height function
 on a hyperbolic Whitney umbrella can have an $A_1^-$ singularity, and
 it on a parabolic Whitney umbrella can have an $A_1^-$ or $A_2$
 singularity.
\end{rem}


\section{Local intersections of a Whitney umbrella with hyperplanes}
 \label{sec:intersection}
 In this section, we study the local intersections of a Whitney umbrella
 with hyperplanes through the Whitney umbrella singularity.

 \colRRR{Let $g:(\BB{R}^2, {\bf 0})\to (\BB{R}^3, {\bf 0})$ be given in
 the normal form of Whitney umbrella, and let $(u_1,v_1)$ and
 $(u_2,v_2)$ be two points on the double point locus $\mathcal{D}$ with
 $g(u_1,v_1)=g(u_2,v_2)$.
 Then we have
 \begin{align}
  \label{eq:eq_double1}
  &u_1=u_2,\\
  \label{eq:eq_double2}
  &0=u_1 v_1-u_2 v_2-\frac{b_3}6({v_1}^3-{v_2}^3)-\frac{b_4}{24}({v_1}^4-{v_2}^4)+\cdots,\\
  \label{eq:eq_double3}
  &0=
  \sum_{i+j=2}^3\frac{a_{ij}}{i!j!}({u_1}^i {v_1}^j-{u_2}^i{v_2}^j)+\cdots.
 \end{align}
 From \eqref{eq:eq_double1} and \eqref{eq:eq_double2}, we have
 \begin{equation}
  \label{eq:eq_double4}
   u_1=-\frac{b_3}6({v_1}^2+v_1 v_2+{v_2}^2)-\frac{b_4}{24}({v_1}^3+{v_1}^2v_2+v_1{v_2}^2+{v_2}^3)+\cdots.
 \end{equation}
 From \eqref{eq:eq_double1}, \eqref{eq:eq_double3}, and
 \eqref{eq:eq_double4}, we have
 \begin{equation}
  \label{eq:eq_double5}
  v_2=-v_1+\frac{a_{11}b_3-3a_{03}}{3a_{02}}{v_1}^2+\cdots.
 \end{equation}
 It follows from \eqref{eq:eq_double4} and \eqref{eq:eq_double5} that
 the parameterization of the double point locus $\mathcal{D}$ in the
 domain is given by
 \begin{equation}
  \label{eq:double_point_locus}
   \gamma(v)=\left(-\frac{b_3}6v^2+\frac{a_{11}b_3-3a_{03}}{18a_{02}}v^3+O(v^4),\,v\right).
 \end{equation}}
   
\begin{lmm}
 \label{lem:double_point_locus}
 \colCC{Let $h:(\BB{R}^2,{\bf 0})\to \BB{R}$ be a height function on
 $g$.
 Then $h(\gamma(v))$ is flat \colRRRR{at $v=0$} \textR{$($that is, all derivatives of $h(\gamma(v))$ vanish at $v=0)$} or} 
 the degree of the first term of $h(\gamma(v))$ is even.
\end{lmm}
\begin{proof}
 \colCC{If} $(u_1,v_1)$ and $(u_2,v_2)$ are two points on $\mathcal{D}$
 with  $g(u_1,v_1)=g(u_2,v_2)$
 \colCC{, then} 
%
 we have $h(\gamma(v_1))=h(\gamma(v_2))$.
 \colRRR{We assume that $h(\gamma(v))$ is not flat \colRRRR{at $v=0$}. 
 This means 
 \colRRRR{
 \[
  \left.\frac{d^m}{dv^m}h(\gamma(v))\right|_{v=0}\ne0
 \]}
 for some $m>1$.
 If $m$ is odd, the sign of $h$ changes near $v=0$ and we have
 contradiction. }
 \end{proof}
 \begin{thm}
  \label{thm:main2}
  Let $g:(\BB{R}^2,{\bf 0})\to (\BB{R}^3,{\bf 0})$ be given in the
  normal form of Whitney umbrella, and let $\mathcal{P}_{{\bf w}_0}$
  denote the hyperplane through the origin with normal vector ${\bf
  w}_0\in S^2$.
  \colCC{Moreover, let $h:(\BB{R}^2,{\bf 0})\to\BB{R}$ be the height
  function on $g$ 
  \textR{in the direction ${\bf w}_0$.}}
  Then the following hold:
  \begin{enumerate}
   \item Suppose that $\mathcal{P}_{{\bf w}_0}$ contains
         the tangent \colRR{line} to the image of the double point locus
         $\mathcal{D}$ under $g$ at the origin\colCC{, 
         and suppose that $h(\gamma(v))$ is not flat \colRRRR{at $v=0$},
         where $\gamma$ is the parameterization of $\mathcal{D}$ in
         \eqref{eq:double_point_locus}.} 
         \begin{enumerate}
          \item Suppose that 
                $h$ is non-singular at $(0,0)$.
                Then $h^{-1}(0)$ and $\mathcal{D}$, in the parameter
                space, have $2k$-point contact $(k\geq 2)$ at
                $(0,0)$. Moreover, the local intersection of $g$ with
                $\mathcal{P}_{{\bf w}_0}$ has 
                \textR{a singularity of type $A_{2k}$ in
                $\mathcal{P}_{{\bf w}_0}$ as \textR{a 
                curve in the plane $\mathcal{P}_{{\bf w}_0}$.}} 
          \item Suppose that $h$ 
                is singular at $(0,0)$.
                Then one of the local components of $h^{-1}(0)$ which is
                not the $u$-axis and $\mathcal{D}$, in the parameter
                space, have $(2k+1)$-point contact $(k\geq 1)$ at
                $(0,0)$. Moreover, the local intersection of $g$ with
                $\mathcal{P}_{{\bf w}_0}$ has 
                \textR{a singularity of type $D_{2k+3}$ in
                $\mathcal{P}_{{\bf w}_0}$ as \textR{a 
                curve in the plane $\mathcal{P}_{{\bf w}_0}$.}}
         \end{enumerate}
   \item Suppose that $\mathcal{P}_{{\bf w}_0}$ does not contain
         the tangent \colRR{line} to image of $\mathcal{D}$ under $g$, and that $h$ 
         has an $A_k^\pm$ singularity
         at $(0,0)$.
         Then the local intersection of $g$ with $\mathcal{P}_{{\bf
         w}_0}$ 
         \textR{has a singularity of type $A_{k+2}$ in
         $\mathcal{P}_{{\bf w}_0}$ as a curve in the plane
         $\mathcal{P}_{{\bf w}_0}$.}
  \end{enumerate}
 \end{thm}
 \begin{proof}
  Now we regard $S^2$ as the set of unit vectors in $\BB{R}^3$ and we
  write ${\bf w}_0=(x_0,y_0,z_0)$.
  
  \noindent$(1)$\quad
  \colCC{From \eqref{eq:double_point_locus}, the tangent line to the
  the image of the double point locus $\mathcal{D}$, as a space curve,
  under $g$ at the origin is given by
  \begin{equation}
   \label{eq:tangent_to_D}
   \{(-b_3,0,3a_{02})t\,;t\in\BB{R}\}.
   \end{equation}
  }
  From assumption and 
  \colCC{\eqref{eq:tangent_to_D}}, we have
  \begin{equation}
   \label{eq:contain}
   3a_{02}z_0-b_3x_0=0.
  \end{equation}
  Consider the parameterization of the double point locus
  $\mathcal{D}$ given by \eqref{eq:double_point_locus}.

  \colC{We start with the case (a).}
  We have
  \begin{equation}
   \label{eq:contact}
   h(\gamma(v))=\frac16(3a_{02}z_0-b_3x_0)v^2+O(v^3).
  \end{equation}
  From \eqref{eq:contain} and \eqref{eq:contact}, $h^{-1}(0)$ and
  $\mathcal{D}$ have more than 
  \colRR{2-point} contact at $(0,0)$.
  \colRR{
  Lemma~\ref{lem:double_point_locus} implies that
  when $h^{-1}(0)$ is non-singular at $(0,0)$, $h^{-1}(0)$ and $\mathcal{D}$ have
  even-point contact at $(0,0)$.}
  Hence, 
  \colRR{it follows} that
  $h^{-1}(0)$ and $\mathcal{D}$ have $2k$-point contact ($k\geq 2$).
  \colC{The double point locus of the standard Whitney umbrella
  $f:(\BB{R}^2,{\bf 0})\to(\BB{R}^3,{\bf 0})$ given by $(u,v)\mapsto(u,u
  v,v^2)$ is parameterized by $(0,v)$.}
  Since $g$ is $\ma$-equivalent to 
  $f$, there is a local diffeomorphism $\phi$ such 
  that $\phi(\gamma(v))=(0,v)$.
  We assume that $\gamma(v)$ is written as the form
  \[
   \gamma(v)=(d_2v^2+\cdots+d_{2k-1}v^{2k-1}+d_{2k}v^{2k}+\cdots,v) \quad (d_i\in\BB{R}).
  \]
  Then $h^{-1}(0)$ is parameterized by 
  \[
   \hat{\gamma}(v)=(d_2v^2+\cdots+d_{2k-1}v^{2k-1}+p_{2k}v^{2k}+\cdots,v)\quad
   (p_i\in\BB{R},\ p_{2k}\ne d_{2k}).
  \]
  We have
  \textR{$\phi(h^{-1}(0))=\{((p_{2k}-d_{2k})v^{2k}+\cdots,v)\,;v\in(\BB{R},0)\}$}
  and thus 
  \[
  f(\phi(h^{-1}(0)))=\{((p_{2k}-d_{2k})v^{2k}+\cdots,(p_{2k}-d_{2k})v^{2k+1}+\cdots,v^2)\,;v\in(\BB{R},0)\}. 
  \]
  This is locally diffeomorphic to $g(h^{-1}(0))$, which is the
  intersection of $g$ with $\mathcal{P}_{{\bf w}_0}$.
  By a suitable change of coordinates in $\BB{R}^3$, $(2k+1)$-jet of
  $f(\phi(\hat{\gamma}(v)))$ is transformed into
  $(0,(p_{2k}-d_{2k})v^{2k+1},v^2)$.
  Therefore, the local intersection of $g$ with $\mathcal{P}_{{\bf
  w}_0}$ has an $A_{2k}$ singularity at the origin.

  \colC{We turn to the case (b).}
  Now we have $x_0=0$.
  It follow form \eqref{eq:contain} that $z_0=0$.
  Hence, ${\bf w}_0=(0,\pm1,0)$ and thus the height function $h$ is
  given by  
  \[
   h=v\left(u+\frac16b_3v^2+\frac1{24}b_4v^3+O(v^4)\right).
  \]
  From this and 
  Lemma \ref{lem:double_point_locus}, the local component 
  $u=-b_3/6v^2-b_4/24v^3+O(v^4)$ of $h^{-1}(0)$ and $\mathcal{D}$ has
  $(2k+1)$-point contact $(k\geq 1)$.
  The rest of the proof can be completed similarly as in the proof of
  (a). 
  
  \noindent $(2)$\quad
  From assumption and 
  \eqref{eq:tangent_to_D}, we have
  \begin{equation}
   \label{eq:not_contain}
   3a_{02}z_0-b_3x_0\ne0.
  \end{equation}
  Firstly, we prove the case \colRR{that} 
  \colC{$k=0$}, that is, $h$ is non-singular at $(0,0)$.
  From \eqref{eq:contact} and \eqref{eq:not_contain}, it follows that
  $h^{-1}(0)$ and $\mathcal{D}$ have 2-point contact at $(0,0)$.
  Therefor, the proof of (a) of (1) implies that the local intersection of
  $g$ with $\mathcal{P}_{{\bf w}_0}$ has an $A_2$ singularity at the
  origin.
  
  Next, 
  \colC{we prove the case \colRR{that} $k\geq 1$.}
  Now we note that $\cos\theta_0\ne0$.
  We set a rotation $T:\BB{R}^3\to\BB{R}^3$ by
  \begin{equation*}
   T=\left(
      \begin{array}{ccc}
      1 & 0 & 0 \\
       0 & \dfrac{\cos\theta_0}{\ma(\theta_0)} &
        \dfrac{a_{11}\cos\theta_0+a_{02}\sin\theta_0}{\ma(\theta_0)}\\[10pt]
       0 & \ \dfrac{-a_{11}\cos\theta_0-a_{02}\sin\theta_0}{\ma(\theta_0)}
       & \dfrac{\cos\theta_0}{\ma(\theta_0)}
     \end{array}\right).
  \end{equation*}
  We obtain
  \begin{equation*}
   \label{eq:intersection_Ak}
    T\circ g|_{h^{-1}(0)}=\Bigl(u,\ma(\theta_0)\sec\theta_0\Bigl(uv+\dfrac{b_3}6v^3+O(u,v)^4\Bigr),0\Bigr),
   \end{equation*}
  which is a plane curve in $xy$-plane.
  From the proof of Theorem \ref{thm:main1}, $h$ can be expressed in
  \begin{equation*}
    h=
    \frac12c_{02}^0\left(v-\tan\theta_0
    u+\sum_{i+j\geq 2}^k P_{ij}u^iv^j\right)^2+Qu^{k+1}+\cdots,
  \end{equation*}
  for some coefficients $P_{ij}$, $Q(\ne0)\in\BB{R}$, where $c_{02}^0$
  is as in \eqref{eq:monge_height_function}.
  We make changes of coordinates
  \begin{equation}
   \label{eq:changes_of_coordinates}
   \hat{v}=v-\tan\theta_0\,u+\sum_{i+j\geq 2}^k P_{ij}u^iv^j,\quad
    \hat{u}=u.
  \end{equation}
  In this new system, $h$ is expressed in
  \begin{equation}
   \label{eq:hh}
   h=
    \frac12c_{02}^0\hat{v}^2+Q\hat{u}^{k+1}+
    \cdots.
  \end{equation}
  From \eqref{eq:changes_of_coordinates}, we have 
  $v=\tan\theta_0\,\hat{u}+\hat{v}+
  \Sigma_{i+j\geq2}^m\hat{P}_{ij}\hat{u}^i\hat{v}^j+\cdots$
  for some coefficients $\hat{P}_{ij}\in\BB{R}$, where $m=k/2$
  (resp. $m=(k+1)/2$) if $k$ is even (resp. odd).
  So we have
  \[
  T\circ
  g|_{h^{-1}(0)}=\Bigl(\hat{u},
  \ma(\theta_0)\sec\theta_0\Bigl(\tan\theta_0\,\hat{u}^2+\hat{u}\hat{v}
  +\sum_{i+j\geq
  3}^{m+1}\bar{P}_{ij}\hat{u}^i\hat{v}^j+\cdots\Bigr),0\Bigr)
  \quad (\bar{P}_{ij}\in\BB{R}).
  \]
  
  \colC{Suppose that $k$ is even.}
  It follows from \eqref{eq:hh} that $h^{-1}(0)$ is locally parameterized by
  $\hat{u}=\alpha t^2$, $\hat{v}=\beta t^{k+1}$,
  for some coefficients $\alpha$ and $\beta$, not both zero.
  Hence, we show that
  \[
  T\circ g|_{h^{-1}(0)}=
  \text{\small{$\left(\alpha t^2,\ma(\theta_0)\sec\theta_0
  \left(
  \underbrace{\alpha^2\tan\theta_0\,t^4
  +\cdots+\alpha^{m+1}\bar{P}_{m+1,0}\,t^{k+2}}_{\mathrm{even~order~terms}}+\alpha\beta 
  t^{k+3}+\cdots\right),0\right)$}},
  \]
  which has an $A_{k+2}$
  singularity at the origin as a plane curve.
  
  \colC{Suppose that $k$ is odd.}
  It follows form \eqref{eq:hh} that $h^{-1}(0)$ is locally parameterized
  by $\hat{u}=t$, $\hat{v}=\pm\sqrt{-2Q/c_{02}^0}\,t^{(k+1)/2}$.
  Hence, we show that
  \[
  T\circ g|_{h^{-1}(0)}=
  \left(t,\ma(\theta_0)\sec\theta_0\left(\tan\theta_0\,t^2+\cdots+\bar{P}_{m,0}t^{\frac{k+1}2}
  \pm\sqrt{-\frac{2Q}{c_{02}^0}}\,t^{\frac{k+3}2}+\cdots\right),0\right)
  \]
  which has an $A_{k+2}^\pm$ singularity at the origin as a plane curve.
 \end{proof}

 \colRRR{In the case $(1)(b)$, the local intersection contains a regular
 curve which is tangent to $x$-axis at the origin. 
 The curvature of the regular curve is considered in \cite{HHNUY2012}.}
 
 \begin{figure}[ht]
   \begin{minipage}[t]{0.33\textwidth}
   \centering
   \includegraphics[width=0.8\textwidth,angle=0,clip]{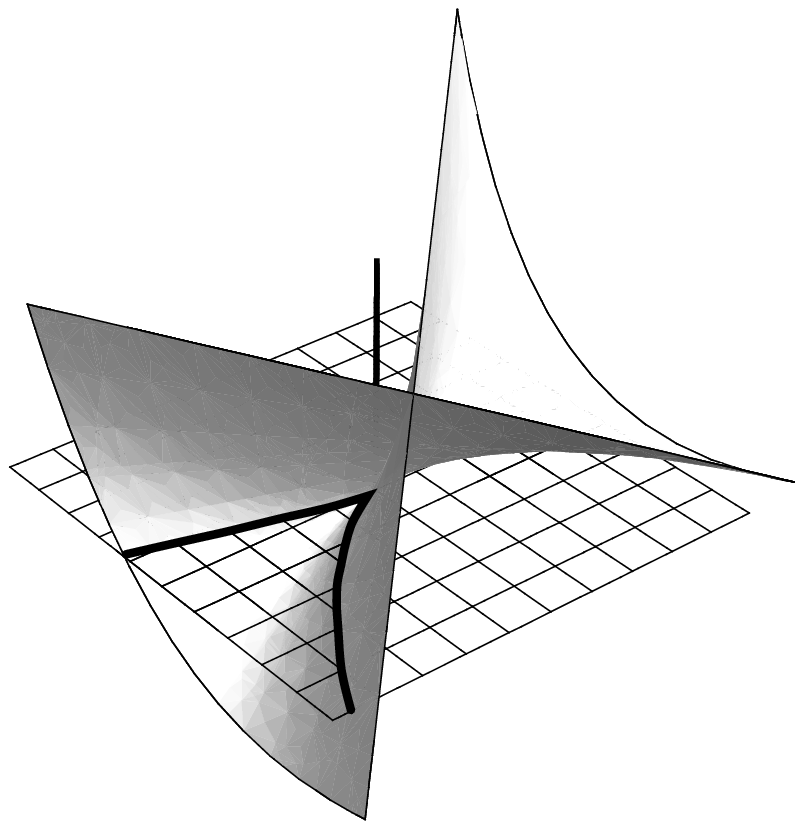}
  \end{minipage}\hfill
  \begin{minipage}[t]{0.33\textwidth}
   \centering
   \includegraphics[width=0.8\textwidth,angle=0,clip]{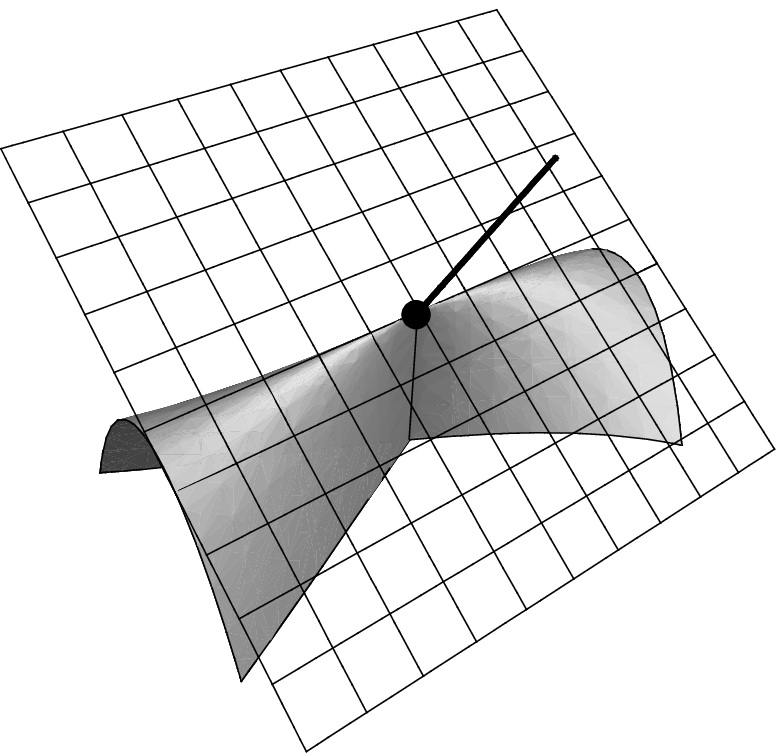}
  \end{minipage}\hfill
  \begin{minipage}[t]{0.33\textwidth}
   \centering
   \includegraphics[width=0.8\textwidth,angle=0,clip]{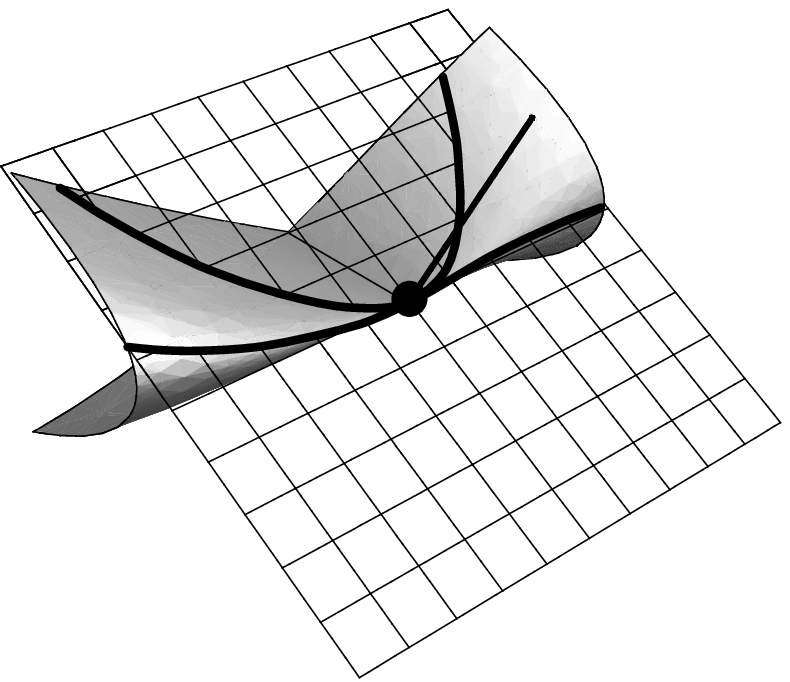}
  \end{minipage}\hfill
  \begin{minipage}[t]{0.33\textwidth}
   \centering
   \includegraphics[width=0.8\textwidth,angle=0,clip]{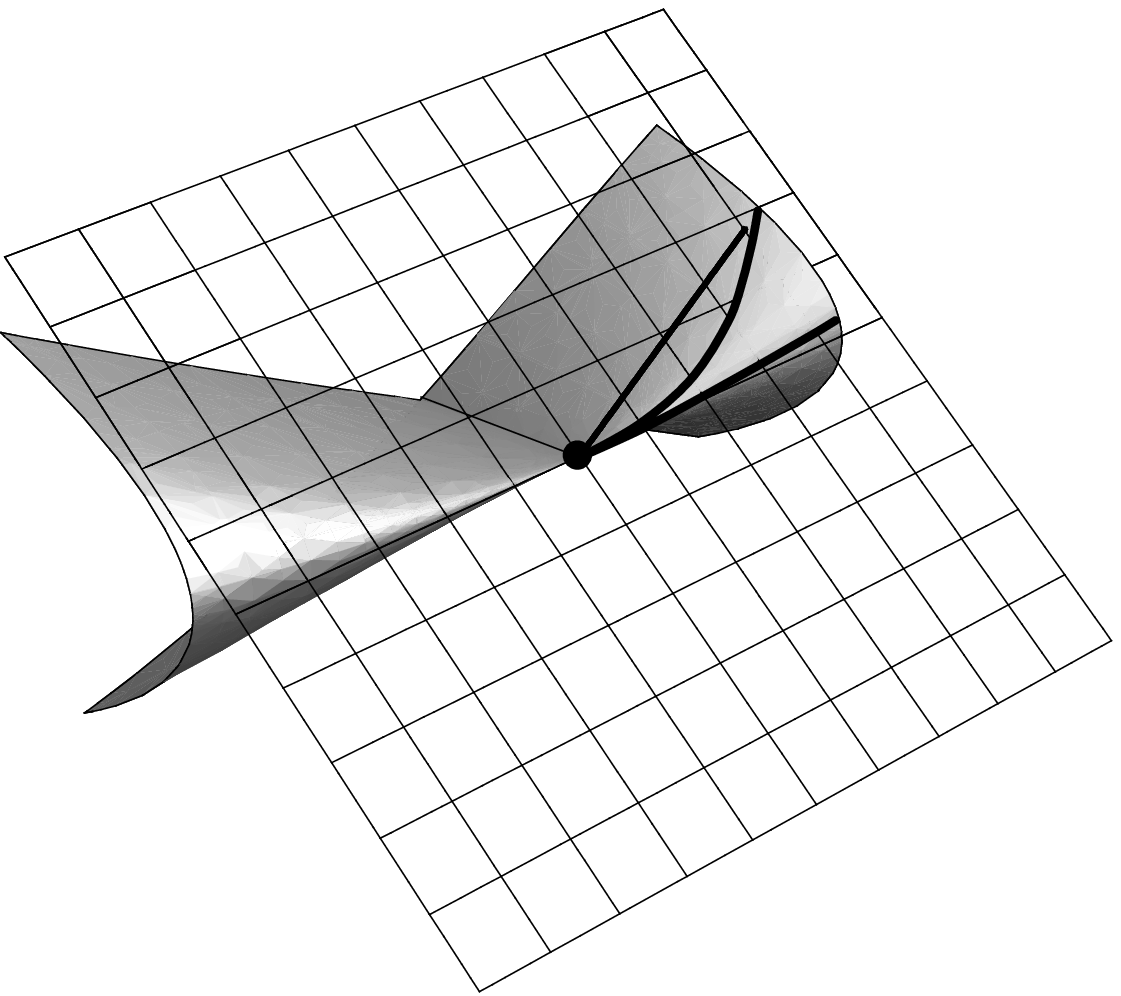}
  \end{minipage}\hfill
  \begin{minipage}[t]{0.33\textwidth}
   \centering
   \includegraphics[width=0.8\textwidth,angle=0,clip]{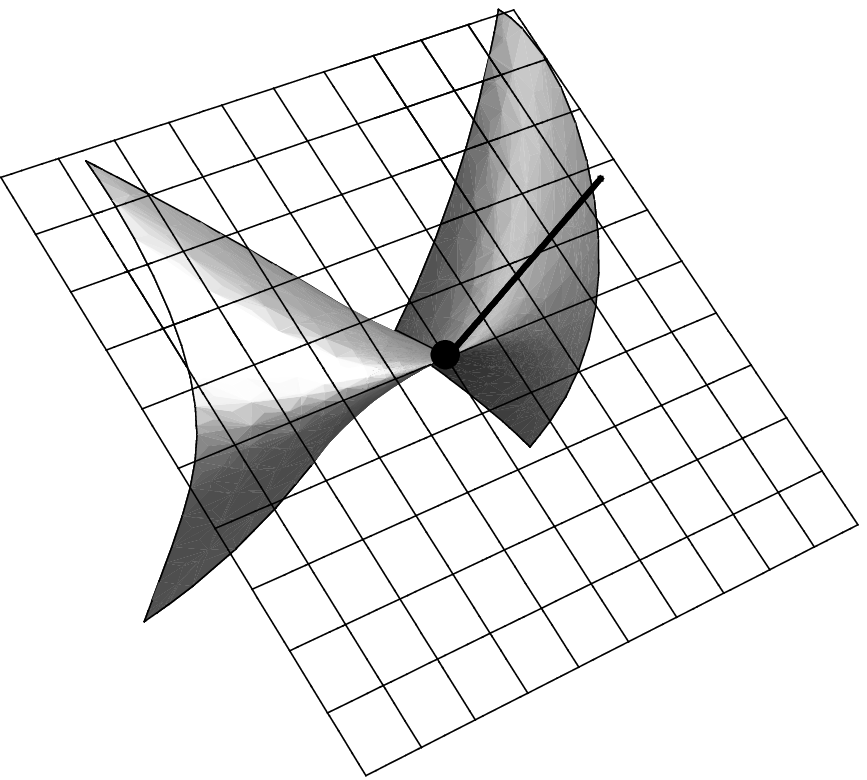}
  \end{minipage}\hfill
  \begin{minipage}[t]{0.33\textwidth}
   \centering
   \includegraphics[width=0.8\textwidth,angle=0,clip]{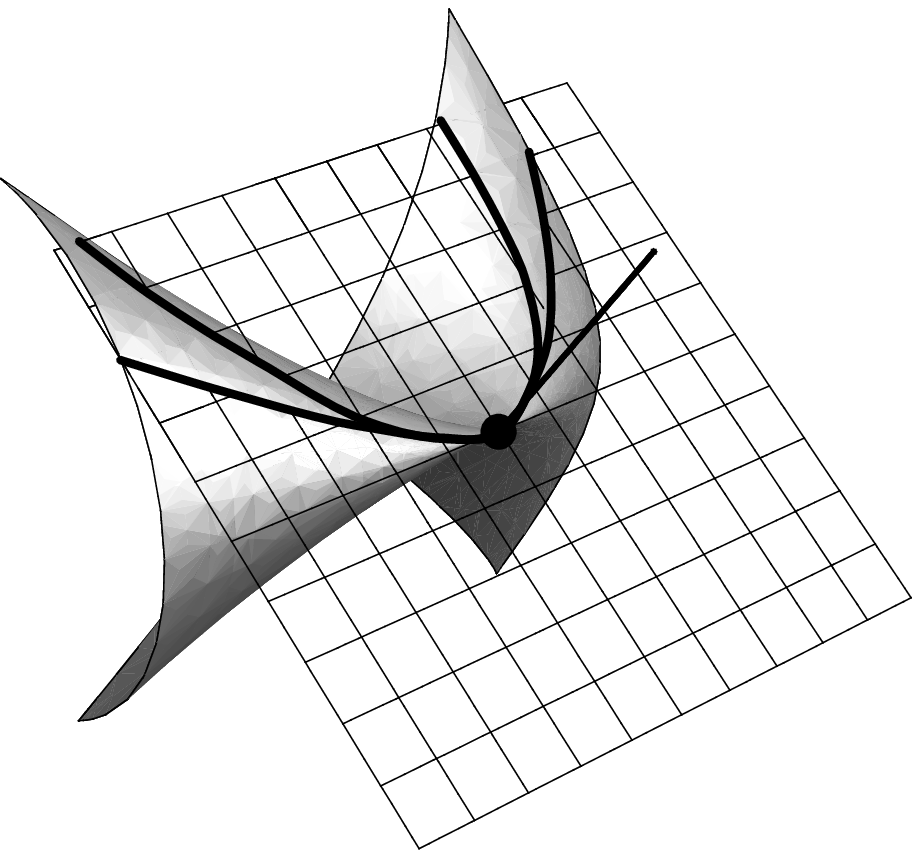}
  \end{minipage}\hfill
    \textR{\caption{Intersection of the Whitney umbrella with the plane $\mathcal{P}_{{\bf w}_0}$ so that
  $h_{{\bf w}_0}$ has a singularity of type $A_0$ (top left), $A_1\p$
  (top center), $A_1\m$ (top right), $A_2$ (bottom left), $A_3\p$
  (bottom center), and $A_3\m$ (bottom right) when
  $\mathcal{P}_{{\bf w}_0}$ does not 
  contain the tangent line to the double point locus on
  the Whitney umbrella.}}
 \end{figure}
 

\end{document}